\documentclass[a4paper,12pt]{amsart}

  \usepackage{amssymb,amsthm,times}
\usepackage{graphicx,color}    
 \usepackage{changes,soul}    

\usepackage{url} 

\newtheorem{theorem}{Theorem}[section]
\newtheorem{corollary}[theorem]{Corollary}
\newtheorem{proposition}[theorem]{Proposition}

\theoremstyle{definition}
\newtheorem{definition}[theorem]{Definition}
\newtheorem{example}[theorem]{Example}
\newtheorem{remark}[theorem]{Remark}

\DeclareMathOperator{\sech}{sech}

\def\e{\mathbb E}
\def\r{\mathbb R}

\def\l{\mathbb L}

\def\d{\mathsf d}
\def\d{\mathbb D}
\begin{document}

\title[Constant curvature curves]{Constant curvature curves in dual affine and dual Lorentz-Minkowski planes}
 
\author{Muhittin Evren Aydin}
\address{Department of Mathematics, Faculty of Science, Firat University, Elazig,  23200 Turkey}
\email{meaydin@firat.edu.tr}
\author{Nursem\.in \c{C}avdar}
\address{Department of Mathematics, Faculty of Science and Art, Tekirdag Namik Kemal University, 59030, Tekirdag, Turkey}
\email{nursemin\_cavdar@hotmail.com}
\author{Mahmut Erg\"ut}
\address{Department of Mathematics, Faculty of Science and Art, Tekirdag Namik Kemal University, 59030, Tekirdag, Turkey}
\email{mergut@nku.edu.tr}
\begin{abstract}
In this paper, we first study invariants of curves parametrized by a real variable in the dual plane $\mathbb{D}^2$ under equiaffine transformations. We then obtain explicit equations for all curves in $\mathbb{D}^2$ whose equiaffine curvature is a dual constant. In particular, we prove that when the equiaffine curvature is a pure real constant, both the real and dual parts of the curve in $\mathbb{D}^2$ are quadratic curves. In addition, we provide a complete classification of spacelike and timelike curves parametrized by a real variable in the dual Lorentz--Minkowski plane $\mathbb{D}^2_1$ whose curvature is a dual constant.
\end{abstract}

\maketitle
\section{Introduction}

A classical problem in differential geometry concerns the study of curves with constant curvature in ambient spaces. The simplest setting for this problem is the plane, since in the Euclidean plane $\e^2$ or Lorentzian plane $\l^2$, any curve is, up to isometries, uniquely determined by its curvature. In this paper, we investigate this problem in the dual affine plane $\d^2$ and dual Lorentzian plane $\d_1^2$.

The study of dual numbers was initiated by William Kingdon Clifford (1845-1879); see \cite{clif}. Recall that a dual number is of the form $x+\varepsilon y$, where $x,y\in \r$ are called its real and dual parts respectively and $\varepsilon^2=0$ with $\varepsilon\neq0$. The set $\d$ of all dual numbers is a commutative ring with unity. The arithmetic operations on $\d$ are known, see for example \cite{vel}. An $n$-dimensional dual space $\d^n$ is a module over the ring $\d$; in fact $\d^n=\r^n \times \varepsilon \r^n$.

Let $l$ be a straight line in $\e^3$ passing through a point $p\in \e^3$ and parallel to a vector $\vec{x}\in \e^3$ with $|\vec{x}|=1$. Eduard Study (1862-1930) \cite{stud} established a correspondence between lines in $\e^3$ and dual vectors in $\d^3$, given by $l \mapsto \widehat{x}=\vec{x}+\varepsilon (p \times \vec{x})$, where the dual part represents the moment vector of the line $l$. Since the real part $\vec{x}$ and dual part $p \times \vec{x}$ are perpendicular, the corresponding dual vector $\widehat{x}$ is unit and hence represents the line $l$ in the dual unit sphere model. Together with this representation, Study’s other foundational contributions gave rise to the systematic investigation of kinematics of curves.

As mentioned, the aim of this paper is to study curves in the dual setting. Indeed, curves in the $3$-dimensional dual space $\d^3$ have been investigated by several authors. Without attempting to provide a complete list, we refer the reader to \cite{ay,lcj,naj,oo,ty,vel}. However, more recently, Rafael L\'opez \cite{lo2} pointed out that the arc-length reparametrization of curves in $\d^3$ commonly used in the literature contains a misformulation. This is particularly significant since the arc-length parameter is the fundamental object for establishing invariants of curves.

More clearly, let $\langle,\rangle$ be the dual metric on $\d^n$, obtained by extending the Euclidean metric of $\e^3$ under the condition $\varepsilon^2=0$. We denote by $|\cdot|$ the norm induced by $\langle,\rangle$. Let $I\subset \r$ be a real interval and $\gamma:I\to\d^3$ be a parametrized curve given by $\gamma(s)=\alpha(s)+\varepsilon \beta(s)$, where $\alpha,\beta:I\to \e^3$ are smooth curves satisfying $|\alpha'(s)|=1$, for all $s\in I$. The commonly used arc-length parameter $\widehat{s}$ is defined by
\begin{equation}\label{introarc}
\widehat{s}=\int^s_0 |\gamma'(u)|du=s+ \varepsilon f(s),
\end{equation}
where $f(s)=\int^s_0 \langle\alpha'(u),\beta'(u)\rangle du.$

L\'opez explicitly showed that the parameter $\widehat{s}$ given in \eqref{introarc} does not define a dual parameter. Indeed, in order for $\widehat{s}$ to be a dual parameter, its real and dual parts must be independent variables; in other words, it must define an open set in $\d\cong \r^2$. In this case, however, the dual part $f(s)$ of the arc-length parameter $\widehat{s}$ depends on the real part $s$. Thus, $\widehat{s}$ traces a curve in $\r^2$ rather than describing a domain and therefore cannot be regarded as a dual parameter. Consequently, this makes not possible to define the derivative of $\gamma$ with respect to $\widehat{s}$ as a dual variable.

On the other hand, for differentiation with respect to a real parameter to be well-defined, the dual part $f(s)$ of the arc-length parameter $\widehat{s}$ must vanish. In addition, the real part $\alpha(s)$ is also required to be regular.

After fixing this misformulation, L\'opez \cite{lo2} established fundamental existence theorems for curves in the dual setting and provided a classification of those with constant curvature and torsion.

Motivated by L\'opez’s comprehensive work \cite{lo2}, we investigate curves in the dual affine plane $\d^2$ and dual Lorentzian plane $\d_1^2$ parametrized by a real variable. The investigation in $\d_1^2$ constitutes a natural extension of \cite{lo2}. However, the study of equiaffine invariants for curves in $\d^2$ requires, in principle, a different approach.

In particular, we introduce equiaffine arc-length for a curve $\gamma(s)=\alpha(s)+\varepsilon \beta(s)$, $s\in I\subset \r$, in $\d^2$ satisfying $(\alpha'(s),\alpha''(s))=1$ given by
$$
\widehat{s}=\int^s_0 \sqrt[3]{(\gamma',\gamma'')}du=s+\frac{\varepsilon}{3}\int^s_0[(\alpha'(u),\beta''(u))+(\beta'(u),\alpha''(u))]du,
$$
where $(,)$ denotes the determinant. As L\'opez pointed out, in order to study curves in $\d^2$ parametrized by a real variable, the dual part of $\widehat{s}$ must vanish, while its real part must be nondegenerate.

In Section \ref{sec3}, by means of a system of ODEs involving the equiaffine curvature, we first provide a characterization of curves $\gamma(s)=\alpha(s)+\varepsilon \beta(s)$ satisfying $\kappa_\gamma=\kappa_\alpha$ (see Theorem \ref{eqk=k}), where $\kappa_\gamma$
and $\kappa_\alpha$ denote the equiaffine curvatures of $\gamma$ and $\alpha$, respectively. We then classified such curves when $\kappa_\gamma$ is a dual constant. An interesting observation is that if $\kappa_\gamma$ is a dual constant, then it is either a pure dual or a pure real constant (see Corollary \ref{eqcurva0}, Theorem \ref{eqcurva01} and Corollary \ref{eqcurva02}).

The study of $\d^3$ endowed with a Lorentzian metric was initiated by U\u{g}urlu ve \c{Ç}alı\c{s}kan \cite{uc} and this space is denoted by $\d_1^3$. See also \cite{get}. Within the study of curves in $\d_1^3$, the same misformulation can also be found in several works \cite{ay,acy,oo}.

In Section \ref{sec4}, we also classify spacelike and timelike curves in $\d_1^2$ whose curvatures coincide with those of their real parts (see Corollary \ref{k=k}). Moreover, we provide a complete classification of spacelike and timelike curves whose curvature is a dual constant (see Theorem \ref{lclass}).

\section{Preliminaries}

Let $\r^2$ be the real affine plane and let $\vec{u},\vec{v}\in \r^2$ be vectors given by $\vec{u}=(u_1,u_2)$ and $\vec{v}=(v_1,v_2)$. We endow $\r^2$ with a fixed area form $\det (\vec{u},\vec{v})=u_1v_2-u_2v_1$. For simplicity, we use $(\cdot,\cdot)$ to denote the determinant. 

Let $(x,y) $ be the canonical coordinates of $\r^2$. An equiaffine transformation $F:\r^2 \to \r^2$ is given by
$$
(x,y) \mapsto F(x,y), \quad F(x,y)=(a_{11}x+a_{12}y+m,a_{21}x+a_{22}y+n), 
$$
where $a_{ij},m,n \in \r$ and $a_{11}a_{22}-a_{12}a_{21}=1.$

Let $\alpha:I\subset \r\to \r^2$ be a non-degenerate smooth curve, given by $\alpha=\alpha(t)$, that is, $(\alpha'(t),\alpha''(t))\neq 0$ for all $t\in I$. Then, there is an equaffine arc-length reparametrization of $\alpha$ given by $$s(t)=\int^t_{t_0}\sqrt[3]{(\alpha'(u),\alpha''(u))}du.$$ With respect to this parameter, we have $(\alpha'(s),\alpha''(s))=1$ for all $s$. Moreover, the equiaffine curvature of $\alpha$ is defined  by $\kappa_\alpha(s)=(\alpha''(s),\alpha'''(s))$.

If $\kappa_\alpha$ is a real constant $r_\alpha \in \r$, then, up to equiaffine transformations of $\r^2$ $\alpha$ is either a parabola ($r_\alpha =0$), an ellipse ($r_\alpha >0$) or a hyperbola ($r_\alpha <0$). More clearly, these curves can be written in the explicit forms $y=\frac12x^2$ and $r_\alpha x^2+r_\alpha^2 y^2=1$ (see \cite{ns}).

We next endow $\r^2$ with the Lorentzian metric $\langle ,\rangle =dx^2-dy^2$. We denote the Lorentz-Minkowski plane by $\l^2=(\r^2,\langle ,\rangle)$.  A vector $\vec{u}\in \l^2$ is called spacelike if $\langle \vec{u},\vec{u}\rangle>0$ or $\vec{u}=0$, timelike if $\langle \vec{u},\vec{u}\rangle>0$ and lightlike if $\langle \vec{u},\vec{u}\rangle=0$ and $\vec{u} \neq 0$. 

Let $M=\mbox{diag}(1,-1)$ and consider the affine transformation
$$
F:\l^2\to \l^2, \quad p\mapsto F(p)=Ap+b,
$$
where $A$ is a $2\times2$ matrix and $b\in \l^2$. The transformation $F$ is called a Lorentzian isometry if it satisfies $A^\top M A=M$. Such Lorentzian isometries can be classified into four types (see \cite{lo1,one} for details).

Let $\vec{u},\vec{v}\in \l^2$ satisfy $\langle \vec{u},\vec{v}\rangle=0$. If $\vec{u}$ is spacelike (resp. timelike), then $\vec{v}$ is timelike (resp. spacelike). Moreover, if $\vec{u}$ is lightlike, then $\vec{v}$ is also lightlike and necessarily linearly dependent with $\vec{u}$.

Let $\alpha:I\subset \r\to \l^2$ a regular curve, namely $\alpha'(t)\neq 0$ for all $t\in I$. The curve $\alpha(t)$ is called spacelike (resp. timelike) curve if the velocity $\alpha'(t)$ is a spacelike (resp. timelike) vector for all $t\in I$. Moreover, a point $\alpha(t)$ is called lightlike if the vector $\alpha'(t)$ is lightlike at $t$.

A spacelike or timelike curve $\alpha(s)\subset \l^2$, $s\in I$, is said to be parametrized by arc-length if $|\alpha'(s)|=\sqrt{|\langle \alpha
',\alpha'\rangle|}=1$, for all $s\in I$. Let $T_\alpha=\alpha'$ be the tangent vector and $N_\alpha$ the corresponding normal vector such that $\langle T_\alpha,T_\alpha\rangle =\delta \in \{-1,1\}$ and $\langle N_\alpha,N_\alpha\rangle =-\delta$. Then, the Frenet formulas for the curve $\alpha$ in $\l^2$ are given by
$$
T'_\alpha=\kappa_\alpha N_\alpha, \quad N'_\alpha=\kappa_\alpha T_\alpha,
$$
where $\kappa_\alpha=-\delta \langle T',N \rangle$ is called the curvature of $\alpha$.

Set $\theta(s)={\textstyle\int^s}\kappa_\alpha (u)du$. Up to isometries of $\l^2$, a spacelike arc-length parametrized curve $\alpha(s)$ is uniquely determined by its curvature and can be written as (see \cite{ccc})
$$
\alpha(s)=\left ({\textstyle\int^s}\cosh \theta(u)du, {\textstyle\int^s}\sinh \theta(u)du \right ).
$$
For timelike curves, the corresponding expression is
$$
\alpha(s)=\left ({\textstyle\int^s}\sinh \theta(u)du, {\textstyle\int^s}\cosh \theta(u)du \right ).
$$

\section{Curves in dual affine plane}\label{sec3}

Let $M_2(\d)$ denote the set of all $2\times 2$ matrices $A=(a_{ij})$ whose entries are dual numbers. The determinant of such a matrix is defined by $\det A=a_{11}a_{22}-a_{12}a_{21}$. We denote by $SL(2,\d)=\{A\in M_2(\d): \det A=1\}$ the the corresponding special linear group over the ring of dual numbers.

In this section, we study the invariant of curves parametrized by a real variable in $\d^2$ under the action of the group $SL(2,\d)$. For this, let $\gamma:I\subset \r \to \d^2$ be a smooth curve given by $\gamma(t)=\alpha(t)+\varepsilon\beta(t)$, where $\alpha,\beta:I\to \r^2$ are smooth functions. We can also represent $\gamma$ in column form
$$
\gamma(t)=\begin{bmatrix}
\gamma_1(t) \\ 
\gamma_2(t)
\end{bmatrix},
$$
where $\gamma_i(t)=\alpha_i(t)+\varepsilon \beta_i(t)$ with $\alpha_i,\beta_i$ smooth real-valued functions on $I$.

If $\alpha\subset \r^2$ is non-degenerate, then a direct calculation yields
$$
(\gamma',\gamma'')=(\alpha',\alpha'')+\varepsilon((\alpha',\beta'')+(\beta',\alpha''))
$$ 
and
$$
(\gamma',\gamma'')^{\frac13}=(\alpha',\alpha'')^{\frac13}+\frac{\varepsilon}{3(\alpha',\alpha'')^{\frac23}}((\alpha',\beta'')+(\beta',\alpha'')).
$$
Therefore, we introduce the following.

\begin{definition}
A curve $\gamma(t)=\alpha(t)+\varepsilon\beta(t)$, $t\in I\subset \r$, is said to be non-degenerate in $\d^2$ if its real part $\alpha$ is non-degenerate in $\r^2$, that is, $(\alpha'(t),\alpha''(t))\neq 0$ for all $t\in I$. Moreover, $\gamma$ is said to be parametrized by equiaffine arc-length if $(\gamma'(t),\gamma''(t))=1$ for all $t\in I$.

\end{definition}

Assume that $s$ is the equiaffine arc-length parameter for the curve $\alpha\subset \r^2$, namely $(\alpha'(s),\alpha''(s))=1$, for all $s\in I\subset \r$. Consider the parameter $\widehat{s}$ given by
$$
\widehat{s}=\int^s_0(\gamma'(u),\gamma''(u))^{\frac13}du=s+\frac{\varepsilon}{3}\int^s_0[(\alpha'(u),\beta''(u))+(\beta'(u),\alpha''(u))]du.
$$
Therefore, for a non-degenerate $\gamma$ to be equiaffinely reparametrized, it is necessary that
$$
(\alpha'(s),\beta''(s))+(\beta'(s),\alpha''(s))=0, \quad s\in I.
$$

\begin{proposition}\label{eqarc}
Let $\gamma(t)=\alpha(t)+\varepsilon \beta(t)$, $t\in I\subset \r$, be a non-degenerate curve in $\d^2$. The curve $\gamma$ admits an equiaffine arc-length reparametrization if and only if 
$$
(\alpha'(t),\beta''(t))+(\beta'(t),\alpha''(t))=0, \quad t\in I.
$$
\end{proposition}

\begin{remark}
In the classical setting $\r^2$, non-degeneracy is a necessary and sufficient condition for the existence of an equiaffine arc-length reparametrization. However, in the dual setting $\d^2$, non-degeneracy is only a necessary condition; it is not sufficient and the extra condition in Proposition \ref{eqarc} must also be satisfied.
\end{remark}

\begin{example}
Let $\alpha(s)=(\sin \frac{t}{2}, -\cos \frac{t}{2})$ and $\beta(s)=(\cos \frac{t}{2}, \sin \frac{t}{2})$, $t\in \r$. Then, for all $t\in \r$, we have 
$$
(\alpha'(t),\alpha''(t))=\frac18, \quad \alpha'(t)=-2\beta''(t), \quad \alpha''(t)=\frac12\beta'(t).
$$
Hence, the curve $\gamma(t)=\alpha(t)+\varepsilon \beta(t)$ is non-degenerate and satisfies the condition in Proposition \ref{eqarc}. By the change of parameter $t=2s$, $s\in \r$, we obtain the equiaffine arc-length reparametrization of $\gamma$, that is,
\begin{equation}\label{exeq}
\gamma(s)=(\sin s, -\cos s)+\varepsilon(\cos s, \sin s),
\end{equation}
which satisfies $(\gamma'(s),\gamma''(s))=1$, for all $s\in \r$.
\end{example}

Let $\gamma(s)=\alpha(s)+\varepsilon \beta(s)$, $s\in I\subset \r$, be parametrized by equiaffine arc-length. Then, $(\gamma',\gamma'')=1$ on $ I$. Differentiating, we obtain $(\gamma'(s),\gamma'''(s))=0$, which implies the existence of a dual-valued smooth function such that $\gamma'''=-\kappa_\gamma \gamma'$. Hence, we have $\kappa_\gamma=(\gamma'',\gamma''')$.

\begin{definition}
Let $\gamma(s)=\alpha(s)+\varepsilon \beta(s)$, $s\in I\subset \r$, be a curve parametrized by equiaffine arc-length in $\d^2$. The function 
$$
\kappa_\gamma :I \to \d, \quad \kappa_\gamma(s)=(\gamma''(s),\gamma'''(s))
$$ 
is called the equiaffine curvature of $\gamma$.
\end{definition}

By a computation, we obtain the following relation between the equiaffine curvatures of $\gamma$ and its real part
\begin{equation}\label{eqrel}
\kappa_\gamma=\kappa_\alpha+\varepsilon[(\alpha'',\beta''')+(\beta'',\alpha''')],
\end{equation}
where $\kappa_\alpha$ denotes the equiaffine curvature of $\alpha$, given by $\kappa_\alpha=(\alpha'',\alpha''')$. 

For example, the curve given in \eqref{exeq} has $\kappa_\gamma=1$. Indeed, it satisfies
$$
\kappa_\alpha=1, \quad \alpha''=-\beta''', \quad \alpha'''=\beta''.
$$

Using relation \eqref{eqrel}, one can address the problem of finding all curves in $\d^2$ which satisfy $\kappa_\gamma=\kappa_\alpha$. This condition also provides a characterization of curves in $\d^2$ whose curvature function is purely real-valued.

\begin{theorem}  \label{eqk=k}
Let $\gamma(s)=\alpha(s)+\varepsilon \beta(s)$, $s\in I\subset \r$, be a curve parametrized by equiaffine arc-length in $\d^2$ and let $\kappa_\gamma$ and $\kappa_\alpha$ denote the equiaffine curvatures of $\gamma$ and its real part, respectively. Then $\gamma$ satisfies $\kappa_\gamma=\kappa_\alpha$ if and only if $\beta'=x\alpha'+y\alpha'',$ where $x$ and $y$ are smooth functions on $I$, satisfying the system of ODEs
\begin{equation} \label{eqk=ksys}
\left\lbrace
\begin{array}{l}
x=-\frac12 y',\\
y'''+4\kappa_\alpha y'+2\kappa'_\alpha y=0.
\end{array}
\right.
\end{equation}
\end{theorem}

\begin{proof}
Assume that $\kappa_\gamma=\kappa_\alpha $. We can write $\beta'$ in terms of the basis $\{\alpha',\alpha''\}$ as
$$
\beta'=x\alpha'+y\alpha'',
$$
where $x$ and $y$ are smooth functions on $I$. Differentiating, we obtain
$$
\beta''=(x'-\kappa_\alpha y)\alpha'+(x+y')\alpha'',
$$
where we have used $\alpha'''=-\kappa_\alpha \alpha'$. Set 
$$
X=x'-\kappa_\alpha y, \quad Y=x+y'.
$$
Differentiating $\beta''=X\alpha'+Y\alpha''$, we get
$$
\beta'''=(X'-\kappa_\alpha Y)\alpha'+(X+Y')\alpha''.
$$
Using the equiaffine arc-length condition in Proposition \ref{eqarc}, we derive
$$
0=(\alpha',\beta'')+(\beta',\alpha'')=2x+y'.
$$
In addition, from \eqref{eqrel} it follows that
$$
0=(\alpha'',\beta''')+(\beta'',\alpha''')=X'-2\kappa_\alpha Y,
$$
or equivalently
$$
(x'-\kappa_\alpha y)'-2\kappa_\alpha(x+y')=0.
$$
Using $y'=-2x$, we simplify the last equation to
$$
y'''+2\kappa'_\alpha y+4\kappa_\alpha y'=0.
$$

Conversely, if $\beta'=x\alpha' + y\alpha''$, where $x$ and $y$ are smooth functions satisfying \eqref{eqk=ksys}, then by reversing the above computations we obtain $(\alpha',\beta'') + (\beta',\alpha'') = 0$ and $(\alpha'',\beta''') + (\beta'',\alpha''') = 0$. Therefore, from \eqref{eqrel} we conclude that $\kappa_\gamma = \kappa_\alpha$. This completes the proof.
\end{proof}

The existence of solutions to \eqref{eqk=ksys} is guaranteed by standard local existence theorems for ODEs. In addition, explicit solutions can be obtained in certain particular cases. For instance, consider the case $\kappa_\gamma=\kappa_\alpha=0$. Then, up to equiaffine transformations of $\r^2$, we have $\alpha(s)=(s,\frac{s^2}{2})$, $s\in I$. Therefore by solving \eqref{eqk=ksys} we derive 
\begin{equation} \label{eqk=ksys00}
\beta(s)=(c_0s^2+c_1 s,-\frac{c_1}{2}s^2+c_2s) +\beta_0, 
\end{equation} 
where $c_0,c_1,c_2 \in \r$ are real constants and $\beta_0\in \r^2$ is a constant vector. Consequently, we have $$(\beta'(s),\beta''(s))=-(c_1^2+2c_0c_2), \quad s\in I,$$ implying that $\beta(s)$ is a straight line if $c_1^2+2c_0c_2=0$. Otherwise $\beta(s)$ is a parabola. 

As a consequence, we have classified all curves in $\d^2$ with vanishing equiaffine curvature.

\begin{corollary}\label{eqcurva0}
Let $\gamma(s)=\alpha(s)+\varepsilon \beta(s)$, $s\in I\subset \r$, be a curve parametrized by equiaffine arc-length in $\d^2$. Its equiaffine curvature $\kappa_\gamma$ vanishes identically if and only if, up to equiaffine transformations of $\r^2$, the real part is given by the parabola $\alpha=(s,\frac{s^2}{2})$ and the dual part $\beta(s)$ is given by \eqref{eqk=ksys00}.
\end{corollary}

The following result provides a classification of curves in $\d^2$ whose equiaffine curvature $\kappa_\gamma$ is a dual constant. 

\begin{theorem}\label{eqcurva01}
Let $\gamma(s)=\alpha(s)+\varepsilon \beta(s)$, $s\in I$, be a curve parametrized by equiaffine arc-length in $\d^2$. Then its equiaffine curvature $\kappa_\gamma$ is a nonzero dual constant $\kappa_\gamma=r_\alpha+\varepsilon m$ with $m,r_\alpha \in\r$ if and only if, up to equiaffine transformations of $\r^2$, one of the following holds
\begin{enumerate}
\item[(a)] If $r_\alpha=0$ and $m\neq 0$, then $\alpha(s)=(s,\frac{s^2}{2})$ and
$$
\beta(s)=(-m\frac{s^3}{6}+c_0s^2+c_1s, -m\frac{s^4}{24}-\frac{c_1}{2}s^2+c_2s)+\beta_0.
$$
\item[(b)] If $r_\alpha >0 $, then necessarily $m=0$ and $\alpha(s)=(\frac{\sin (\sqrt{r_\alpha}s)}{\sqrt{r_\alpha}},-\frac{\cos (\sqrt{r_\alpha}s)}{r_\alpha})$,
$$
\beta(s)=(c_0\cos(\sqrt{r_\alpha}s) + c_1\sin(\sqrt{r_\alpha}s) ,\frac{-c_0}{\sqrt{r_\alpha}}\sin(\sqrt{r_\alpha}s)
+ \frac{c_1}{\sqrt{r_\alpha}}\cos(\sqrt{r_\alpha}s) )+\beta_0.
$$
\item[(c)] If $r_\alpha <0 $, then necessarily $m=0$ and $\alpha(s)=(\frac{\sinh(\sqrt{-r_\alpha}s)}{\sqrt{-r_\alpha}},
-\frac{\cosh(\sqrt{-r_\alpha}s)}{r_\alpha})$,
$$
\beta(s)=(c_0\cosh(\sqrt{-r_\alpha}s)+c_1\sinh(\sqrt{-r_\alpha}s),\,
-\frac{c_0}{\sqrt{-r_\alpha}}\sinh(\sqrt{-r_\alpha}s)
-\frac{c_1}{\sqrt{-r_\alpha}}\cosh(\sqrt{-r_\alpha}s)) + \beta_0,
$$
where $c_0,c_1,c_2 \in \r$ and $\beta_0 \in \r^2$ are constants.
\end{enumerate}
\end{theorem}

\begin{proof}
Assume that the equiaffine curvature of $\gamma$ is a nonzero dual constant, say $\kappa_\gamma=r_\alpha+\varepsilon m$, where $m^2+r_\alpha^2\neq 0$ and $m,r_\alpha \in\r$. Comparing with \eqref{eqrel}, we conclude
\begin{equation}\label{eqconst00}
\kappa_\alpha=r_\alpha, \quad (\alpha'',\beta''')+(\beta'',\alpha''')=m.
\end{equation}
Set $\beta(s)=(x(s),y(s))$, for smooth functions $x$ and $y$ on $I$. Excluding the case both $r_\alpha$ and $m$ are zero, we consider three possibilities: $r_\alpha = 0$ and $m \neq 0$, $r_\alpha >0 $ with $m$ arbitrary and $r_\alpha < 0$ with $m$ arbitrary.
\begin{enumerate}
\item[(a)] Case $r_\alpha = 0$ and $m \neq 0$. Up to equiaffine transformations of $\r^2$, the real part of $\gamma$ is given by $\alpha(s)=(s,\frac{s^2}{2})$. The equiaffine arc-length condition in Proposition \ref{eqarc} yields
\begin{equation}\label{eqconst01}
y''-sx''+x'=0.
\end{equation}
From \eqref{eqconst00}, it follows that $-x'''=m$. Substituting this into \eqref{eqconst01}, the proof of the first item of the thorem is completed.

\item[(b)] Case $r_\alpha >0 $ with $m$ arbitrary. After applying suitable  equiaffine transformations of $\r^2$, the real part of $\gamma$ is given by 
$$
\alpha(s)=\left (\frac{\sin (\sqrt{r_\alpha}s)}{\sqrt{r_\alpha}},-\frac{\cos (\sqrt{r_\alpha}s)}{r_\alpha} \right ).
$$
Since $(\alpha',\beta'')+(\beta',\alpha'')=0$, it follows that 
$$
(x'+y'')\cos (\sqrt{r_\alpha}s)+\left (\sqrt{r_\alpha}y'-\frac{1}{\sqrt{r_\alpha}}x'' \right )\sin (\sqrt{r_\alpha}s)=0.
$$
As the functions $\cos (\sqrt{r_\alpha}s)$ and $\sin (\sqrt{r_\alpha}s)$ are linerly independent, we conclude that 
\begin{equation}\label{eqconst02}
x'+y''=0, \quad r_\alpha y'=x''.
\end{equation}
Also, from \eqref{eqconst00} we have 
$$
(x'''-r_\alpha y'')\cos (\sqrt{r_\alpha}s)-\sqrt{r_\alpha}(x''+y''')\sin (\sqrt{r_\alpha}s)=m.
$$
Considering \eqref{eqconst02}, both coefficients vanish and hence $m=0$. Solving \eqref{eqconst02} then yields the expression for $\beta$ in item (b).
\item[(c)] Case $r_\alpha <0 $ with $m$ arbitrary. This case is analogous to the previous one.
\end{enumerate}
\end{proof}

For the real and dual parts $\alpha$ and $\beta$ given in the second item of Theorem \ref{eqcurva01}, by a suitable choice of the integration constants $c_0$ and $c_1$ there exists an area-preserving transformation of $\r^2$ mapping $\alpha$ to $\beta$. Indeed, consider the $2\times 2$ matrix
$$
A=\begin{bmatrix}
c_1\sqrt{r_\alpha} & -c_0 r_\alpha\\ 
-c_0 & -c_1\sqrt{r_\alpha}
\end{bmatrix},
$$
where $r_\alpha>0$ is a real constant. Then $\beta = A\alpha + \beta_0$. Since
$$
\det(A) = -r_\alpha(c_0^2 + c_1^2),
$$
by choosing $r_\alpha(c_0^2 + c_1^2)=1$, we obtain $\det(A) =-1$. Therefore, the resulting transformation $F(p)=A(p)+\beta_0$ of $\r^2$ is an area-preserving transformation but orientation-reversing.

In contrast, for the third item of Theorem \ref{eqcurva01}, by choosing $$r_\alpha(c_0^2- c_1^2)=1, \quad |c_0|>|c_1|,$$ we can establish an equiaffine transformation of $\r^2$ mapping the real part $\alpha$ to the dual part $\beta$, whose linear part is given by
$$
A=\begin{bmatrix}
c_1\sqrt{-r_\alpha} & -c_0 r_\alpha\\ 
-c_0 & -c_1\sqrt{-r_\alpha}
\end{bmatrix},
$$
where $r_\alpha<0$ is a real constant. 

Consequently, we have the following result.

\begin{corollary} \label{eqcurva02}
Let $\gamma(s)$, $s\in I$, be a curve parametrized by equiaffine arc-length in $\d^2$ with nondegenerate dual part. If the equiaffine curvature of $\gamma$ is a nonzero dual constant, then this constant is necessarily either pure real or pure dual. In addition, if it is a pure real constant $r_\alpha \neq 0$, then, up to equiaffine transformations of $\r^2$, the real and dual parts are the quadratic curves given, respectively, by $r_\alpha x^2+r_\alpha^2 y^2=1$ and
$$
\frac{(x-x_0)^2}{a^2}+r_\alpha \frac{(y-y_0)^2}{a^2}=1, 
$$
for some nonzero real constant $a$ and some point $(x_0,y_0)\in \r^2$.
\end{corollary}

\section{Curves in dual Lorentz-Minkowski plane} \label{sec4}

In this section, we introduce the Frenet formulas for curves parametrized by a real variable in $\d_1^2$ and then provide a classification of such curves whose curvature is a dual constant.

Let $\gamma:I\subset \r \to \d_1^2$ be a smooth curve given by $\gamma(t)=\alpha(t)+\varepsilon\beta(t)$. We say that $\gamma \subset \d_1^2$ is regular if $\gamma'(t)\neq 0$ for all $t\in I$. Moreover, the causal character of $\gamma \subset \d_1^2$ is defined to be the causal character of its real part $\alpha \subset \l^2$. 

Hence, we can describe all lightlike curves in $\d_1^2$. Assume that $\gamma\subset \d_1^2$ is a lightlike curve, that is a regular curve whose real part is lightlike. Since all lightlike curves $\alpha(t)\subset\l^2$ are well-known, a lightlike curve $\gamma(t)$ in $\d_1^2$ can therefore be written in the form
\begin{equation}\label{lig}
\gamma(t)=p+t\vec{v}+\varepsilon \beta(t), 
\end{equation}
where $p\in \l^2$ is some point, $\vec{v}\in \l^2$ is a lightlike vector and $\beta(t)\subset\l^2$ is an arbitrary curve. If, in addition, $\beta(t)\subset\l^2$ is also a lightlike curve whose velocity $\beta'$ parallel to $\vec{v}$, then the function $\langle \gamma',\gamma'\rangle=0$ on $I$.

In what follows, we provide a criterion for a curve $\gamma$ to admit a reparametrization by arc-length. The proof is analogous to the Euclidean case and hence we refer to \cite[Proposition 2.1]{lo2}.

\begin{proposition}\label{larc}
Let $\gamma(t)=\alpha(t)+\varepsilon \beta(t)$, $t\in I\subset \r$, be a curve. The curve $\gamma$ has a reparametrization by arc-length if and only if $\alpha\subset\l^2$ is a regular curve and the function $\langle \alpha', \beta'\rangle $ vanishes on $I$.
\end{proposition}

We also remark that a timelike curve is always regular. Hence, when $\gamma\subset \d_1^2$ is timelike, the regularity condition in Proposition \ref{larc} is directly satisfied.

Assume that $\gamma(s)=\alpha(s)+\varepsilon \beta(s)$, $s\in I\subset \r$, is parametrized by arc-length and $\beta' \neq 0$ on $I$. By Proposition \ref{larc}, we have 
$$
\langle \alpha'(s),\alpha'(s)\rangle=\delta \in \{-1,1\}, \quad \langle \alpha'(s), \beta'(s)\rangle =0, \quad s\in I,
$$ 
implying that $\beta(s)\subset\l^2$ is spacelike (resp. timelike) whenever $\alpha(s)\subset\l^2$ is timelike (resp. spacelike).


In the following result, we give a description of straight lines in $\d^2_1$.

\begin{proposition}\label{stra}
Let $\gamma(s)=\alpha(s)+\varepsilon \beta(s)$, $s\in I\subset \r$, be an arc-length parametrized curve with $\beta' \neq 0$ on $I$. If $\gamma \subset \d_1^2$ is a straight line, that is $\gamma''=0$, then $\alpha\in \l^2$ is a spacelike (resp. timelike) straight line and $\beta\in \l^2$ is a timelike (resp. spacelike) straight line. 
\end{proposition}

\begin{proof}
The result follows immediately from the arc-length condition $\langle \alpha', \beta'\rangle =0$.
\end{proof}

\begin{remark}
The converse of Proposition \ref{stra} is not always true. Let $\gamma(s)=\alpha(s)+\varepsilon \beta(s)$, $s\in I\subset \r$, be an arc-length parametrized curve so that $|\gamma'(s)|=|\alpha'(s)|=1$. If $\alpha\subset \l^2$ is a straight line, then $\alpha(s)=p+s\vec{v}$, where $\vec{v}\in \l^2$ is a unit vector and $p\in\l^2$ is a point. By Proposition \ref{larc} we have $\langle \vec{v}, \beta'\rangle =0$, implying that $\beta'$ is parallel to a direction orthogonal to $\vec{v}$. Hence, there is a vector $\vec{w}\in \l^2$ orthogonal to $\vec{v}$ such that $\beta(s)=f(s)\vec{w}+q$, for some smooth function $f$ on $I$ and some point $q\in\l^2$. Then the curve
$$
\gamma(s)=p+\varepsilon q + s\vec{v}+\varepsilon f(s)\vec{w}
$$
describes a straight line (i.e. a $1$-dimensional affine subspace) in $\d^2_1$ if and only if $f''=0$ on $I$. Consequently, although $\alpha$ and $\beta$ are straight lines, $\gamma$ is not necessarily a straight line.
\end{remark}

From now on, we assume that $\alpha''\neq 0$ in the arc-length parametrization of $\gamma \subset \d_1^2$ in order to introduce the notion of curvature.

\begin{definition}\label{lcurv}
Let $\gamma(s)=\alpha(s)+\varepsilon \beta(s)$, $s\in I\subset \r$, be an arc-length parametrized curve. We call $T_\gamma=\gamma'$ the unit tangent vector of $\gamma $. If $\alpha''(s) \neq 0$ for all $s\in I$, then we call $$\kappa_\gamma=|\gamma''|=\sqrt{|\langle\gamma'',\gamma''\rangle|}$$ the curvature of $\gamma$ and $$N_\gamma=\frac{1}{\kappa_\gamma}\gamma''$$ the unit normal vector of $\gamma $. We also say that $\gamma$ has vanishing curvature if $\gamma''(s) = 0$ for all $s\in I$. 
\end{definition}

\begin{remark}\label{lrcurv}
In Definition \ref{lcurv}, the curvature $\kappa_\gamma$ of the curve $\gamma$ is not well defined when $\alpha''=0$ but $\gamma''\neq 0$. This issue arises because the norm of a dual vector cannot be consistently defined in this case.
\end{remark}

Let $\delta \in \{-1,1\}$. It is direct to see that $\langle\gamma',\gamma'\rangle =\delta$ whenever $\langle\alpha',\alpha'\rangle =\delta$. Moreover, for $\alpha\in \l^2$ since $\alpha'$ is orthogonal to $\alpha''$ we have $\langle\alpha'',\alpha''\rangle=-\delta |\langle\alpha'',\alpha''\rangle|$. Then, a direct computation shows that $\langle\gamma'',\gamma''\rangle =-\delta|\langle\gamma'',\gamma''\rangle|$, that is, $\langle N_\gamma,N_\gamma\rangle=-\delta$. Consequently, we have the Frenet formulas
$$
T'_\gamma=\kappa_\gamma N_\gamma, \quad N'_\gamma=\kappa_\gamma T_\gamma.
$$

Let $\{T_\alpha,N_\alpha\}$ be the Frenet frame along $\alpha$ and $\kappa_\alpha$ denote the curvature of $\alpha.$ The following result establishes explicit relations between the Frenet apparatus of $\gamma$ and that of its real part $\alpha$.

\begin{proposition}\label{ltnk}
Let $\gamma(s)=\alpha(s)+\varepsilon \beta(s)$, $s\in I\subset \r$, be an arc-length parametrized curve. Then we have
\begin{eqnarray*}
T_\gamma&=&T_\alpha+\varepsilon \beta',\\[2mm]
N_\gamma&=&N_\alpha+\varepsilon \frac{\delta\langle \beta'',T_\alpha\rangle}{\kappa_\alpha}  T_\alpha  ,\\[2mm]
\kappa_\gamma&=&\kappa_\alpha+\varepsilon(-\delta\langle \beta'',N_\alpha\rangle).
\end{eqnarray*}
\end{proposition}

\begin{proof}
The first equality follows directly from the definition of $T_\gamma$. Let $\langle\alpha',\alpha'\rangle =\delta$. Then $\langle\alpha'',\alpha''\rangle=-\delta |\langle\alpha'',\alpha''\rangle|$. Hence,
$$
\kappa_\gamma=\sqrt{|\langle \gamma'',\gamma''\rangle|}=\sqrt{|\langle \alpha'',\alpha''\rangle|+\varepsilon(-2\delta\langle \alpha'',\beta''\rangle)}=\sqrt{|\langle \alpha'',\alpha''\rangle|}+\varepsilon\frac{-\delta\langle \alpha'',\beta''\rangle}{\sqrt{|\langle \alpha'',\alpha''\rangle|}},
$$
which proves the third equality of the proposition. On the other hand, the decomposition of $\beta''$ with respect to the frame $\{T_\alpha,N_\alpha\}$ is given by
\begin{equation} \label{expbe}
\beta''+\delta \langle \beta'',N_\alpha\rangle N_\alpha=\delta \langle \beta'',T_\alpha\rangle T_\alpha.
\end{equation}
Moreover, we have
$$
N_\gamma=\frac{1}{\kappa_\gamma}(\alpha''+\varepsilon \beta'').
$$
Using the division rule for dual numbers together with \eqref{expbe} in the above expression, we obtain the second equality of the proposition. 
\end{proof}

With the third equality of Proposition \ref{ltnk}, we can find all curves in $\d_1^2$ satisfying $\kappa_\gamma=\kappa_\alpha$.

\begin{corollary}\label{k=k}
Let $\gamma(s)=\alpha(s)+\varepsilon \beta(s)$, $s\in I\subset \r$, be an arc-length parametrized curve with $\beta'\neq 0$ on $I$. Let also $\kappa_\gamma$ and $\kappa_\alpha$ denote the curvatures of $\gamma$ and its real part, respectively. The curve $\gamma$ satisfies $\kappa_\gamma=\kappa_\alpha \neq 0$ if and only if it admits one of the following parametrizations: 
\begin{enumerate}
\item[(a)] If $\gamma$ is spacelike, then
$$
\gamma(s)=\left ({\textstyle\int^s}\cosh \theta(u)du, {\textstyle\int^s}\sinh \theta(u)du \right )+\varepsilon m\left ({\textstyle\int^s}\sinh \theta(u)du, {\textstyle\int^s}\cosh \theta(u)du \right ).
$$
\item[(b)] If $\gamma$ is timelike, then
$$
\gamma(s)=\left ({\textstyle\int^s}\sinh \theta(u)du, {\textstyle\int^s}\cosh \theta(u)du \right )+ \varepsilon m\left ({\textstyle\int^s}\cosh \theta(u)du, {\textstyle\int^s}\sinh \theta(u)du \right ),
$$
where $m\in \r$, $m>0$.
\end{enumerate} 
\end{corollary}

\begin{proof}
We do the proof for the case $\gamma$ is spacelike. The proof for the other case is analogous. Assume that $\gamma$ is spacelike; then its real part $\alpha$ is also spacelike. Hence, we choose the following paramterization of $\alpha$:
$$
\alpha(s)=\left ({\textstyle\int^s}\cosh \theta(u)du, {\textstyle\int^s}\sinh \theta(u)du \right ), \quad \kappa_\alpha(s)=\theta'(s).
$$
Set $\beta(s)=(x(s),y(s))$, where $x$ and $y$ are smooth functions on $I$. Using the arc-length condition $\langle \alpha',\beta'\rangle=0$, we have
\begin{equation}\label{lxy00}
x'\cosh \theta-y'\sinh \theta=0.
\end{equation}
Assume $x$ and $y$ are linear functions. Since the functions $\{\sinh \theta,\cosh \theta\}$ are linearly independent, it follows that $x'=0$ and $y'=0$. Consequently, $x$ and $y$ must be non-linear functions. Moreover, by the hypothesis $\kappa_\gamma=\kappa_\alpha$, we have $\langle \beta'',N_\alpha\rangle=0$, or equivalently,
\begin{equation}\label{lxy01}
x''\sinh \theta-y''\cosh \theta=0.
\end{equation}
From \eqref{lxy00}, we have $x'=\tanh \theta y'$ and $x''=\theta'\sech^2 \theta y' +\tanh \theta y''$. Substituting these expressions into \eqref{lxy01}, we derive $y''=\theta'\tanh \theta y'$. Solving this differential equation yields
$$
x'(s)=m\sinh \theta(s), \quad y'(s)=m \cosh \theta(s),
$$
where $m>0$ is a real constant. Integrating these identities completes the proof.
\end{proof}

In view of Proposition \ref{ltnk}, the curvature $\kappa_\gamma$ is purely real-valued if and only if $\langle \beta'',N_\alpha\rangle = 0$. In particular, when $\kappa_\gamma\neq 0$, this is equivalent to $\kappa_\gamma=\kappa_\alpha\neq 0$ and Corollary \ref{k=k} provides a complete description of such curves.

Moreover, in the spacelike case of Corollary \ref{k=k} the real part $\alpha$ is a spacelike curve whereas the dual part $\beta$ is timelike, and the roles are reversed in the timelike case. Since isometries of $\mathbb{L}^2$ preserve the causal character, there is no isometry of $\l^2$ mapping the real part $\alpha$ to the dual part $\beta$, or conversely.

As a concluding result of this section, we provide a complete classification of curves in $\d_1^2$ with nonzero constant curvature.

\begin{theorem} \label{lclass}
Let $\gamma(s)=\alpha(s)+\varepsilon \beta(s)$, $s\in I\subset \r$, be a spacelike curve in $\d^2_1$ parametrized by arc-length. The curve $\gamma$ has constant curvature $r_\gamma=r_\alpha+\varepsilon m$, where $r_\alpha,m\in \r$ and $r_\alpha\neq 0$, if and only if, up to the isometries of $\l^2$, the real part $\alpha$ is given by
$$
\alpha(s)=r_\alpha\left ( \sinh \frac{s}{r_\alpha} ,\cosh \frac{s}{r_\alpha}\right), 
$$
and the dual part satisfies
\begin{eqnarray*}
\beta(s)&=&\Big(-m r_\alpha^2\sinh\frac{s}{r_\alpha}
+ (ms+n) r_\alpha \cosh\frac{s}{r_\alpha},\\
&&-m r_\alpha^2\cosh\frac{s}{r_\alpha}+ (ms+n) r_\alpha \sinh\frac{s}{r_\alpha}\Big)
+\beta_0,
\end{eqnarray*}
where $\beta_0 \in \l^2$ is a constant vector.
\end{theorem}

\begin{proof}
By hypothesis, the curvature satisfies $\kappa_\gamma = r_\gamma=r_\alpha+\varepsilon m$, where $r_\alpha,m\in \r$ and $r_\alpha\neq 0$. Using the third equality of Proposition \ref{ltnk}, we obtain $\kappa_\alpha=r_\alpha$ and 
\begin{equation}\label{betanm}
\langle \beta'',N_\alpha\rangle=-m.
\end{equation}
Since $\alpha$ is spacelike, up to isometries of $\l^2$, we write
$$
\alpha(s)=r_\alpha\left ( \sinh \frac{s}{r_\alpha}  ,\cosh \frac{s}{r_\alpha} \right).
$$
From the arc-length condition $\langle \beta', T_\alpha \rangle = 0$, there exists a smooth function $f(s)$ on $I$ such that $\beta' = f N_\alpha$. Differentiating, we obtain $\beta'' = f \kappa_\alpha T_\alpha + f' N_\alpha$.
Substituting this expression into \eqref{betanm} yields $f' = m$, or equivalently, $f(s)=ms+n$, $n\in \r$. Since 
$$
N_\alpha=\left ( \sinh \frac{s}{r_\alpha}  ,\cosh \frac{s}{r_\alpha} \right),
$$
integrating the relation $\beta'=fN_\alpha$, we obtain the formula for $\beta$ in the statement of the theorem. The converse of the statement follows directly.
\end{proof}

\begin{remark}
Following the proof of Theorem \ref{lclass}, when $\gamma(s)=\alpha(s)+\varepsilon \beta(s)$ is timelike, we obtain the following parametrizations for the real and dual parts
$$
\alpha(s)=r_\alpha\left ( \cosh \frac{s}{r_\alpha}  ,\sinh \frac{s}{r_\alpha} \right), 
$$
and
\begin{eqnarray*}
\beta(s)&=&\Big( - m r_\alpha^2\cosh\frac{s}{r_\alpha}+(ms+n) r_\alpha\sinh\frac{s}{r_\alpha},\\
&& - m r_\alpha^2\sinh\frac{s}{r_\alpha} +(ms+n) r_\alpha \cosh\frac{s}{r_\alpha}\Big)
+\beta_0, \quad \beta_0 \in \l^2.
\end{eqnarray*}
\end{remark}

\subsection*{Ethics declarations}

The authors have no conflict of interest to declare that are relevant to the content of this article. No data were used to support this study.

\subsection*{Acknowledgment}

The authors would like to express their sincere appreciation to Professor Rafael L\'opez for his valuable suggestion to investigate the differential geometry of curves in the dual spaces. The second author is supported by TÜBİTAK.


\end{document}